\documentclass[12pt]{amsart}
\usepackage[mathscr]{eucal}
\usepackage{amssymb,amsmath}
\usepackage{amsfonts}
\usepackage{amscd}
\usepackage{a4}

\usepackage{color}

\theoremstyle{plain}
\newtheorem{thm}{Theorem}[section]
\newtheorem{prp}[thm]{Proposition}
\newtheorem{cor}[thm]{Corollary}
\newtheorem{lem}[thm]{Lemma}

\newtheorem{defn}[thm]{Definition}

\theoremstyle{remark}
\newtheorem{prob}[thm]{Problem}

\numberwithin{equation}{section}

\newcommand{\epse}{\varepsilon}

\newcommand{\phe}{\varphi}

\newcommand{\To}{\longrightarrow}

\begin{document}

\title
{Extension operators on balls and on spaces of finite sets}

\author{Antonio Avil\'{e}s, Witold Marciszewski}
\address{Universidad de Murcia, Departamento de Matem\'{a}ticas, Campus de Espinardo 30100 Murcia, Spain.} \email{avileslo@um.es}
\address{Institute of Mathematics\\
University of Warsaw\\ Banacha 2\newline 02--097 Warszawa\\
Poland}
\email{wmarcisz@mimuw.edu.pl}

\date{\today}
\subjclass[2010]{Primary 46B26, 46E15, 54C35, 54H05}
\keywords{$C(K)$, extension operator}

\thanks{The first author was supported by MINECO and FEDER (Project MTM2011-25377). Research of the second author was partially supported by the National Science Center research grant DEC-2012/07/B/ST1/03363}

\begin{abstract} We study extension operators between spaces $\sigma_n(2^X)$ of subsets of $X$ of cardinality at most $n$. As an application, we show that if $B_H$ is the unit ball of a nonseparable Hilbert
space $H$, equipped
with the weak topology, then, for any $0<\lambda<\mu$, there is no
extension operator $T: C(\lambda B_H)\to C(\mu B_H)$.
\end{abstract}

\maketitle

\section{Introduction}\label{intro}
Given a compact space $K$ (that we assume to be Hausdorff), by $C(K)$ we denote the Banach space of
continuous real-valued functions on $K$, equipped with the standard
supremum norm. When $L$ is a closed subspace of $K$, Tietze's extension theorem asserts that every $f\in C(L)$ can be extended to a continuous function $\hat{f}\in C(K)$ defined on all $K$. It is however a delicate problem whether the assignment $f\mapsto \hat{f}$ can be done in a linear and continuous way.

\begin{defn}
A bounded linear operator $T: C(L)\to C(K)$ is called an extension operator if, for
every $f\in C(L)$, $Tf$ is an extension of $f$.
\end{defn}

When $L$ is metrizable, the Borsuk-Dugundji extension theorem~\cite[II.4.14]{LT} provides such extension operator $T:C(L)\To C(K)$ which is moreover \emph{regular}: $T$ is positive, preserves constant functions and $\|T\|=1$. We are interested in the possible norms for extension operators, so following Corson and Lindenstrauss \cite{CL}, for $L\subset K$, we consider:
$$\eta(L,K) = \inf \left\{\|T\| : T:C(L)\To C(K) \text{ is an extension operator}\right\}.$$

When $L$ is nonmetrizable it might be the case that there is no extension operator at all, and in this case we agree that $\eta(L,K) = \infty$. We focus on two particular examples of compact spaces.

For a set $X$ and a natural number $n$, we have a compact space
$$\sigma_n(2^X) = \{\chi_A \in \{0,1\}^X : |A|\leq n\}$$
where $\chi_A$ denotes the characteristic function of the set $A$, $\chi_A(x) = 1$ if $x\in A$, and $\chi_A(x) = 0$ if $x\not\in A$. It is well known that any space $\sigma_n(2^X)$ is scattered and is an Eberlein compact spaces, i.e., is homeomorphic to a weakly compact subset of a Banach space. In particular, any Radon measure on $\sigma_n(2^X)$ is purely atomic, and $\sigma_n(2^X)$ is a Fr\'echet topological space.

We shall study extension operators for the inclusion $\sigma_m(2^X) \subset \sigma_n(2^X)$ when $n<m$. Such extension operators always exist (cf. \cite[Prop.\ 3.1]{Ma}), but the optimal norm depends on the cardinality on $X$. When $X$ is countable, by the aforementioned Borsuk-Dugundji theorem, we can get extension operators of norm one. When $|X|\geq \aleph_n$ we shall see in Section~\ref{sectionbigX} that any extension operator has to be somewhere close to a canonical form, and from this we get that
$$\eta(\sigma_m(2^X),\sigma_n(2^X))  = \sum_{k=0}^m \binom{n}{k}\binom{n-k-1}{m-k}$$
On the other hand, if $|X|=\aleph_1$, we can use the special structure of $\omega_1$ to improve the norm of extension operators and we get
$$\eta(\sigma_m(2^{\aleph_1}),\sigma_n(2^{\aleph_1}))  = 2n - 2m +1$$
The result for large $X$ relies on a combinatorial lemma on the existence of \emph{free sets} for set-valued maps \cite{ER,BM}, while the result for $\aleph_1$ requires a weakening of such lemma which is valid on any uncountable set that will be proved in Section~\ref{sectionpre}.

 The other example that we consider is that of balls in a Hilbert space, endowed with their weak topology. Corson and Lindenstrauss \cite[Proposition 1]{CL} showed that, for
the unit ball $B_H$ of a nonseparable Hilbert space $H$ and any
$0<\lambda<\mu$, there exists no weak-continuous retraction $r:
\mu B_H\to \lambda B_H$, i.e., a map $r$ such that $r(x)=x$ for
every $x\in \lambda B_H$. We will show 
the following stronger result:

\begin{thm}\label{balls} Let $H$ be a nonseparable Hilbert space  and $B_H$ be the unit ball of
$H$ equipped with the weak topology. Then, for any
$0<\lambda<\mu$, there is no extension operator $T: C(\lambda
B_H)\to C(\mu B_H)$.
\end{thm}

This result is obtained as an application of the computation of  the number $\eta(\sigma_m(2^{\aleph_1}),\sigma_n(2^{\aleph_1}))$ mentioned above.

Our investigations are somehow related to the following known general problem:

\begin{prob}
For which compact spaces $K$ there exists a zerodimensional compact space $L$ such that the spaces $C(K)$ and $C(L)$ are isomorphic?
\end{prob}

P.\ Koszmider \cite{Ko} constructed the first example of compact space $K$ without the above property, another examples were given by G.\ Plebanek \cite{Pl}, and recently by A.\ Aviles and P.\ Koszmider  \cite{AK}. On the other hand, by Milutin's Theorem all metrizable compacta $K$ poses this property. It is not known whether this holds true for the larger class of Eberlein compact spaces, some partial results in this direction contains the paper \cite{AA}. Even the following concrete question is still open

 \begin{prob}
 Let $B_H$ be the unit ball of a nonseparable Hilbert space  $H$ equipped with the weak topology. Does there exist a zerodimensional compact space $L$ such that the spaces $C(B_H)$ and $C(L)$ are isomorphic?
 \end{prob}
 
 Isomorphisms between the spaces $C(K)$ are often constructed using the Pe\l czy\'n\-ski decomposition method - a technique based on the factorizations of function spaces, cf.\ \cite{Se}. One method of obtaining such factorizations is to use the extension operators, which motivated our investigations of such operators for subsets of balls of Hilbert spaces.

\section{Preliminaries}\label{sectionpre}

For a compact space $K$, by $M(K)$ we denote the space of all Radon measures on $K$, which can be identified with the dual space $C(K)^*$. $B_{M(K)}$ stands for the unit ball of $M(K)$, we will always consider this ball equipped
with the $weak^*$ topology inherited from $C(K)^*$. For a point $x\in K$, $\delta_x$ denotes the Dirac measure on $K$ concentrated at $x$. When we have an extension operator $T:C(L)\To C(K)$, we have an associated continuous function $\phe_T:K \To \|T\| B_{M(L)}$ given by $\phe_T(x) = T^\ast(\delta_x)$ with the key property that $\phe_T(y) = \delta_y$ when $y\in L$. The function $\phe_T$ can be viewed as \emph{generalized retraction} that sends each point of $K$ to a measure on $L$ (instead of a point of $L$).

For a set $X$ and $n\in\omega$, we use the standard notations
\begin{eqnarray*}
\left[X\right]^{n\ \ } &=& \{A\subset X: |A|=n\}\\
\left[X\right]^{\leq n} &=& \{A\subset X: |A|\leq n\}\\
\left[X\right]^{<n} &=& \{A\subset X: |A|<n\}
\end{eqnarray*}

We will use the following combinatorial lemma, cf. \cite[Lemma
1.1]{ER}, \cite[Lemma 3.2]{BM}

\begin{lem}\label{omega_n} Let $n\in\omega$,
$X$ be a set of cardinality $\ge\aleph_n$, and $S:
[X]^n\to [X]^{<\omega}$ be an arbitrary map. Then there exists
$A\in[X]^{n+1}$ such that, for every $a \in A$, we have
$a\notin\phe(A\setminus \{a\})$.
\end{lem}

The following is an equivalent reformulation, in the way that we will actually use:

\begin{lem}\label{omega_p} Let $m\leq p<\omega$,
$X$ be a set with $|X|\geq \aleph_{p-1}$, and $S:
[X]^{\leq m}\to [X]^{<\omega}$ be a map such that $S(A)\cap A =\emptyset$ for every $A$. Then there exists
$Y\in[X]^{p}$ such that $S(A)\cap Y =\emptyset$ for all $A\in [Y]^{\leq m}$.
\end{lem}

\begin{proof}
Apply Lemma~\ref{omega_n} to $n=p-1$, and the function $S':[X]^n\To [X]^{<\omega}$ given by $S'(A) = \bigcup_{B\in [A]^{\leq m}}S(B)$. 
\end{proof}

We shall also need a version of this lemma that holds for any uncountable set.

\begin{lem}\label{omega_1} 
Let $X$ be an uncountable set, $n<\omega$, and $S:[X]^{<\omega}\To [X]^{<\omega}$. Then there exists $Z = \{z_1,\ldots,z_n\}\in [X]^{n}$ such that $z_j \not\in S(\{z_i : i\in I\})$ whenever $j<\min(I)$ or $j>\max(I)$.
\end{lem}
\begin{proof}
We start the construction with a countable infinite set $Y_1\subset X$. Then inductively, we choose a countable infinite set $Y_j\subset X$ for $j=2,\ldots,n$ such that $$ (\star)\ \ Y_j\cap \bigcup\left\{ S(A) : A\in \left[\bigcup_{i<j}Y_i\right]^{<\omega}\right\} = \emptyset.$$ 
Then, by reverse induction (starting for $j=n$ and then $n-1,n-2,\ldots$) we choose $z_j\in Y_j$ such that
$$(\star\star)\ \ z_j \not\in \bigcup\left\{S(A) : A\in \left[\{z_{j+1},\ldots,z_n\}\right]^{<\omega}\right\}.$$
Then, $(\star)$ guarantees the statement of the lemma when $j>\max(I)$, and $(\star\star)$ when $j<\min(I)$.

\end{proof}

\section{A canonical extension operator between spaces $\sigma_m(2^X)$}\label{sectionbigX}

\begin{thm}\label{naturalextension}
Consider $m<n<\omega$ and a set $X$. Then, we have an extension operator $T_X:C(\sigma_m(2^X))\To C(\sigma_n(2^X))$ given by
$$ T_X(f)(\chi_A) = \sum_{B\in [A]^{\leq m}} (-1)^{m-|B|}\binom{|A| - |B|-1}{m-|B|} f(\chi_B)$$
if  $|A|>m$, and by $T_X(f)(\chi_A) = f(\chi_A)$  if $|A|\leq m$.\\
\end{thm}

The proof of Theorem~\ref{naturalextension} will require a number of technical lemmas showing some combinatorial identities. For any natural numbers $p,q,r,s,t$ such that $p\le\min(q,r,s,t), t\le s$ we define

\begin{eqnarray*} \label{Theta}
\Theta(p,q,r,s,t) = \sum_{i=0}^p (-1)^{q-i}\binom{r}{i}\binom{s-i}{t-i}.
\end{eqnarray*}

Observe that
\begin{eqnarray}
\label{Theta0} \Theta(p,q+1,r,s,t) = -\Theta(p,q,r,s,t)
\end{eqnarray}

\begin{lem}\label{Theta}
For any natural numbers $p,q,r,s,t$ such that $p\le\min(q,r,s,t), t\le s$ we have the following identities
\begin{eqnarray}
\label{Theta1} \Theta(p,q,r+1,s+1,p) &=& \Theta(p,q,r,s,p),\\
\label{Theta2} \Theta(p+1,q+1,p+1,s+1,t) &=& -\Theta(p,q,p,s,t) \quad \mbox{for } p\le t-1.
\end{eqnarray}
\end{lem}

\begin{proof}
We start with the first identity

\begin{eqnarray*}
\Theta(p,q,r+1,s+1,p) = \sum_{i=0}^p (-1)^{q-i}\binom{r+1}{i}\binom{s+1-i}{p-i} \\ = (-1)^{q}\binom{r+1}{0}\binom{s+1}{p}  
+ \sum_{i=1}^p (-1)^{q-i}\left(\binom{r}{i-1} + \binom{r}{i}\right)\binom{s+1-i}{p-i}\\ = (-1)^{q}\binom{s+1}{p} + (-1)^{q-1}\binom{r}{0}\binom{s+1-1}{p-1}\\
+  \sum_{i=1}^{p-1} (-1)^{q-i}\binom{r}{i}\left(\binom{s+1-i}{p-i} - \binom{s+1-i-1}{p-i-1}\right)\\ + (-1)^{q-p}\binom{r}{p}\binom{s+1-p}{p-p}\\
= (-1)^{q}\left(\binom{s+1}{p} - \binom{s}{p-1}\right) + \sum_{i=1}^{p-1} (-1)^{q-i}\binom{r}{i}\binom{s-i}{p-i}\\
+(-1)^{q-p}\binom{r}{p}\binom{s-p}{p-p}\\
= (-1)^{q}\binom{r}{0}\binom{s}{p} + \sum_{i=1}^{p} (-1)^{q-i}\binom{r}{i}\binom{s-i}{p-i}\\
=  \sum_{i=0}^{p} (-1)^{q-i}\binom{r}{i}\binom{s-i}{p-i} = \Theta(p,q,r,s,p)
\end{eqnarray*}
 
 The calculations for the identity (\ref{Theta2}) are very similar

\begin{eqnarray*}
\Theta(p+1,q+1,p+1,s+1,t) = \sum_{i=0}^{p+1} (-1)^{q+1-i}\binom{p+1}{i}\binom{s+1-i}{t-i} \\ = -(-1)^{q}\binom{s+1}{t}  
- \sum_{i=1}^{p} (-1)^{q-i}\left(\binom{p}{i-1} + \binom{p}{i}\right)\binom{s+1-i}{t-i}\\ 
+(-1)^{q-p} \binom{s-p}{t-p-1}
= -(-1)^{q}\binom{s+1}{t} - (-1)^{q-1}\binom{s}{t-1}\\
-  \sum_{i=1}^{p-1} (-1)^{q-i}\binom{p}{i}\left(\binom{s+1-i}{t-i} - \binom{s-i}{t-i-1}\right)\\ -  (-1)^{q-p}\binom{s+1-p}{t-p} + (-1)^{q-p} \binom{s-p}{t-p-1}\\
= - (-1)^{q}\binom{s}{t} - \sum_{i=1}^{p-1} (-1)^{q-i}\binom{p}{i}\binom{s-i}{t-i} - (-1)^{q-p}\binom{s-p}{t-p}\\
= - \sum_{i=0}^{p} (-1)^{q-i}\binom{p}{i}\binom{s-i}{t-i} = -\Theta(p,q,p,s,t)
\end{eqnarray*}
\end{proof}

For any natural numbers $k,l,m$ such that $l\ge 1, m\le k+l$ we define
\begin{eqnarray*}
\Phi(k,l,m) = \sum_{i=0}^m (-1)^{m-i}\binom{k+l}{i}\binom{m+l-i-1}{m-i}.
\end{eqnarray*}

\begin{lem}\label{Phi} For any natural numbers $k,l,m$ such that $l\ge 1, k<m\le k+l$, we have
\begin{eqnarray}
\label{Phi1} \Phi(k,l,k) = 1,\\
\label{Phi2} \Phi(k,l,m) = 0.
\end{eqnarray}
\end{lem}

\begin{proof} To prove the formula (\ref{Phi2}) we will also need to show that

\begin{eqnarray}
\label{Phi3} \Theta(k+l,k+l,k+l,n+l-1,n) = 0 \quad \mbox{for any } n\ge k+l.
\end{eqnarray}

We will prove all formulas by induction on $l$. For $l=1$ in formula (\ref{Phi1}) we have

\begin{eqnarray*}
\Phi(k,1,k) = \sum_{i=0}^k (-1)^{k-i}\binom{k+1}{i}\binom{k+1-i-1}{k-i} = \sum_{i=0}^k (-1)^{k-i}\binom{k+1}{i} \\= 1-1 - \sum_{i=0}^k (-1)^{k+1-i}\binom{k+1}{i} = 1 - \sum_{i=0}^{k+1} \binom{k+1}{i}(-1)^{k+1-i}\\
 = 1 - (1-1)^{k+1} = 1.
\end{eqnarray*}

For $l=1$  the formula (\ref{Phi3}) has the form

\begin{eqnarray*}
\Theta(k+1,k+1,k+1,n+1-1,n) = 
\sum_{i=0}^{k+1} (-1)^{k+1-i}\binom{k+1}{i}\binom{n-i}{n-i}\\
 = \sum_{i=0}^{k+1} \binom{k+1}{i}(-1)^{k+1-i}
 = (1-1)^{k+1} = 0.
\end{eqnarray*}

For $l=1$ the only possible value of $m$ is $k+1$, so for (\ref{Phi2}) we have $\Phi(k,1,k+1) = \Theta(k+1,k+1,k+1,k+1+1-1,k+1) =0$.

To complete the inductive step we will use Lemma \ref{Theta}. By (\ref{Theta1}), for $k\le p \le k+l$ we have

\begin{eqnarray}\label{Phi4}
\Phi(k,l+1,p) = \Theta(p,p,k+l+1,p+l,p)\\
\nonumber = \Theta(p,p,k+l,p+l-1,p) = \Phi(k,l,p).
\end{eqnarray}

In particular, $\Phi(k,l+1,k) =  \Phi(k,l,k)$. From (\ref{Theta2}) we obtain, for $n\ge k+l+1$,

\begin{eqnarray}
 \Theta(k+l+1,k+l+1,k+l+1,n+l,n)\\
\nonumber  = - \Theta(k+l,k+l,k+l,n+l-1,n)
\end{eqnarray}

which completes the proof of (\ref{Phi3}). To complete the proof of (\ref{Phi2}) we should consider two cases. If $m\le k+l$ then $\Phi(k,l+1,m) =  \Phi(k,l,m)$ by (\ref{Phi4}). If $m= k+l+1$ then $\Phi(k,l+1,k+l+1)=\Theta(k+l+1,k+l+1,k+l+1,k+2l+1,k+l+1)=0$ by (\ref{Phi3}).
\end{proof}

For any natural numbers $k,m,s$ such that $s\ge k$, we put $j=\min(k,m)$ and we define
\begin{eqnarray*}
\Psi(k,m,s) = \sum_{i=0}^j (-1)^{k-i}\binom{m}{i}\binom{s-i}{k-i}.
\end{eqnarray*}

\begin{lem}\label{Psi} For any natural numbers $k,m,s$ such that $s\ge k$, we have
\begin{eqnarray}\label{Psi1}
\Psi(k,m,s) &=& (-1)^k\binom{s-m}{k}\quad \mbox{if } s\ge k+m,\\
\label{Psi2}
\Psi(k,m,s) &=& 0\quad \mbox{if } k\ge 1, m\leq s< k+m
\end{eqnarray}
\end{lem}

\begin{proof}
We will start with the formula (\ref{Psi1}) and use the induction on $m$. For $m=0$, $j=0$ and our formula obviously holds true. Assume that $\Psi(k,m,s) = (-1)^k\binom{s-m}{k}$ for any $k,s$ such that $s\ge k+m$. To prove the inductive step we shall consider two cases. If $k\le m$ then $\min(k,m)=\min(k,m+1)=k$, and by (\ref{Theta1})
\begin{eqnarray}
\Psi(k,m+1,s)= \Theta(k,k,m+1,s,k)  = \Theta(k,k,m,s-1,k)\\
\nonumber = \Psi(k,m,s-1) = (-1)^k\binom{s-1-m}{k} = (-1)^k\binom{s-(m+1)}{k}.
\end{eqnarray}
If $k> m$ then $\min(k,m)=\min(k,m+1)=m$, and by (\ref{Theta0}) and (\ref{Theta2})
\begin{eqnarray}\label{Psi3}
\Psi(k,m+1,s)= \Theta(m+1,k,m+1,s,k)\\
\nonumber= - \Theta(m+1,k+1,m+1,s,k)  = \Theta(m,k,m,s-1,k)\\
\nonumber = \Psi(k,m,s-1) = (-1)^k\binom{s-1-m}{k} = (-1)^k\binom{s-(m+1)}{k}.
\end{eqnarray}

We will split the proof of (\ref{Psi2}) into two cases.

If $k\le m$ then  $\Psi(k,m,s) = \Phi(m-s+k-1,s-k+1,k)$. Observe that $s-k+1\ge 1$, $m\ge k > m-s+k-1$ since $k\leq s$, $k\leq m$ and $m\leq s$. Therefore by formula (\ref{Phi2}) from Lemma \ref{Phi} we have $\Phi(m-s+k-1,s-k+1,k)=0$.

If $k > m$ then $\Psi(k,m,s) = \sum_{i=0}^m (-1)^{k-i}\binom{m}{i}\binom{s-i}{k-i}$ and we will prove our formula by induction on $s-k$. Observe that $m\ge 1$ since $k\leq s< k+m$. If $s=k$ then $$\Psi(k,m,s) = \sum_{i=0}^m (-1)^{k-i}\binom{m}{i}=(-1)^{k-m}(1-1)^m=0.$$
The inductive step follows from the formula $\Psi(k,m,s)=\Psi(k,m-1,s-1)$, cf. (\ref{Psi3}).
\end{proof}

The next lemma shows that the function $T_X$ of Theorem~\ref{naturalextension} is well defined.

\begin{lem}
For any set $X$, positive integers $m,n$, $m<n$, and a continuous function $f\in C(\sigma_m(2^X))$, the function $T_X(f)$ defined in Theorem~\ref{naturalextension} is continuous on $\sigma_n(2^X)$.
\end{lem}

\begin{proof}
The space $\sigma_{n}(2^X)$, being Eberlein
compact, is a Fr\'echet topological space. Therefore, it is enough
to show that, for every sequence $(\chi_{A_k})$ of points of
$\sigma_{n}(2^X)\setminus \sigma_{m}(2^X)$, converging
to a point $\chi_B\in \sigma_{n}(2^X)$ we have
$T_X(f)(\chi_{A_k})\to T_X(f)(\chi_B)$. Without loss of generality we
may assume that $B\subset A_k$ for all $k$, and all sets $A_k$ have the same cardinality $r$. One can easily verify
that, for every $D\subset B$ and $C_k\subset (A_k\setminus B)$, we
have $\chi_{C_k\cup D}\to \chi_D$. Let $p=|B|$. 
For any $D\in [B]^{\le m}$, let $j_D=\min(m-|D|,r-p)$.
We shall consider two cases. If $p\le m$, i.e., $\chi_B\in \sigma_{m}(2^X)$ then we have
\begin{eqnarray*}
T_X(f)(\chi_{A_k}) &=& \sum_{B_k\in[A_k]^{\le m}}(-1)^{m-|B_k|}\binom{r - |B_k|-1}{m-|B_k|}f(\chi_{B_k})\\  &=& \sum_{C_k\in
[A_k\setminus B]^{\le m-p}}(-1)^{m-|C_k|-p}\binom{r - |C_k|-p-1}{m-|C_k|-p}f(\chi_{C_k\cup B})\\ &+&
\sum_{D\subsetneqq B}\sum_{E_k\in [A_k\setminus
B]^{\le m-|D|}}(-1)^{m-|E_k|-|D|}\binom{r - |E_k|-|D|-1}{m-|E_k|-|D|}f(\chi_{E_k\cup D})\\
&\stackrel{k\to\infty}{\longrightarrow}&
f(\chi_{B})\sum_{i=0}^{m-p}(-1)^{m-p-i}
\binom{r-p}{i}\binom{r-p-i-1}{m-p-i}
\\&+&\sum_{D\subsetneqq
B}f(\chi_{D})\sum_{i=0}^{j_D}(-1)^{m-|D|-i}
\binom{r-p}{i}\binom{r -|D|-i-1}{m-|D|-i}.
\end{eqnarray*}
By Lemma \ref{Phi}
\begin{eqnarray*}
\sum_{i=0}^{m-p}(-1)^{m-p-i}
\binom{r-p}{i}\binom{r-p-i-1}{m-p-i} = \Phi(m-p,r-m,m-p) = 1.
\end{eqnarray*}
For any $D\subsetneqq B$, $m-|D|\ge 1$ and $r-p < r-|D|-1 < (m-|D|)+(r-p)$. Hence, by formula (\ref{Psi2}) from Lemma \ref{Psi}
\begin{eqnarray*}
\sum_{i=0}^{j_D}(-1)^{m-|D|-i}
\binom{r-p}{i}\binom{r -|D|-i-1}{m-|D|-i} = \Psi(m-|D|,r-p,r-|D|-1) = 0.
\end{eqnarray*}
Therefore, $T_X(f)(\chi_{A_k})\longrightarrow f(\chi_{B})$.

In the second case, when $p>m$,  we have
\begin{eqnarray*}
T_X(f)(\chi_{A_k}) &=& \sum_{B_k\in[A_k]^{\le m}}(-1)^{m-|B_k|}\binom{r - |B_k|-1}{m-|B_k|}f(\chi_{B_k})\\
 &=&\sum_{D\in [B]^{\le m}}\sum_{E_k\in [A_k\setminus
B]^{\le m-|D|}}(-1)^{m-|E_k|-|D|}\binom{r - |E_k|-|D|-1}{m-|E_k|-|D|}f(\chi_{E_k\cup D})\\
&\stackrel{k\to\infty}{\longrightarrow}&
\sum_{D\in [B]^{\le m}}f(\chi_{D})\sum_{i=0}^{j_D}(-1)^{m-|D|-i}
\binom{r-p}{i}\binom{r -|D|-i-1}{m-|D|-i}.
\end{eqnarray*}
For any $D\in [B]^{\le m}$,  $(m-|D|)+(r-p)\le r-|D|-1$. Hence, by formula (\ref{Psi1}) from Lemma \ref{Psi}
\begin{eqnarray*}
\sum_{i=0}^{j_D}(-1)^{m-|D|-i}
\binom{r-p}{i}\binom{r -|D|-i-1}{m-|D|-i} = \Psi(m-|D|,r-p,r-|D|-1)\\ = (-1)^{m-|D|}\binom{p-|D|-1}{m-|D|} = (-1)^{m-|D|}\binom{|B|-|D|-1}{m-|D|}.
\end{eqnarray*}
Therefore, 
\begin{eqnarray*}
T_X(f)(\chi_{A_k})\longrightarrow \sum_{D\in [B]^{\le m}}(-1)^{m-|D|}\binom{|B|-|D|-1}{m-|D|}f(\chi_{D}) = f(\chi_{B}).
\end{eqnarray*}
\end{proof}

It is straightforward that $T_X$ is linear and bounded operator, so the proof of Theorem~\ref{naturalextension} is completed.

If we have a subset $Y\subset X$, then we have a natural isometric embedding $$e_Y^X:C(\sigma_m(2^Y)) \To C(\sigma_m(2^X))$$ given by $e_Y^X(f)(\chi_A) = f(\chi_{A\cap Y})$, and also a natural restriction operator $$r_Y^X: C(\sigma_m(2^X)) \To C(\sigma_m(2^Y))$$ given by $r_Y^X(f)(\chi_A) = f(\chi_{A})$. These two operators have norm one. If we have an arbitrary extension operator $T:C(\sigma_m(2^X))\To C(\sigma_n(2^X))$, then a new \emph{restricted} extension operator $$T|_Y:C(\sigma_m(2^Y))\To C(\sigma_n(2^Y))$$ is induced as $T|_Y = r_Y^X \circ S \circ  e_Y^X$. This new extension operator satisfies $\|T|_Y\|\leq \|T\|$, so in particular we have:

\begin{prp}\label{smallerset}
 If $|Y|\leq |X|$, then $\eta(\sigma_m(2^Y),\sigma_n(2^Y)) \leq \eta(\sigma_m(2^X),\sigma_n(2^X))$.
\end{prp}

\begin{thm}\label{somewherenatural}
Fix $m<n\leq p <\omega$ and $\varepsilon>0$. If $T:C(\sigma_m(2^X))\To C(\sigma_n(2^X))$ is any extension operator and $|X|\geq \aleph_{p-1}$, then there exists $Y\in [X]^p$ such that $\|T|_Y-T_Y\|<\varepsilon$, where $T|_Y$ is the restricted operator, and $T_Y$ is the operator from Theorem~\ref{naturalextension}.
\end{thm}

\begin{proof}
We fix $\varepsilon>0$ and we take a much smaller $$\varepsilon'<\frac{\varepsilon}{2^{(p+2)^2}}.$$ Consider the measure-valued function $\phe_T:\sigma_n(2^X) \To M(\sigma_m(2^X))$ associated to the operator $T$, given by $\phe_T(\chi_A) = T^\ast(\delta_{\chi_A})$. For every $A\in [X]^{\leq n}$ choose a finite set $S^1_A\subset X\setminus A$ such that 
$$|\phe_T(\chi_A)|\{\chi_C \in \sigma_m(2^X) : C\not\subset A\cup S^1_A\}<\varepsilon'$$

For $\varepsilon>0$, and reals $a,b$, we will write $a\stackrel{\varepsilon}{\sim}b$ if $|a-b|<\varepsilon$. For every $A,B\in [X]^{\le m}$ define
\begin{eqnarray}\label{eta1_1}
O_B = \{\chi_C\in \sigma_{m}(2^X)\colon B\subseteq C \} \,,\hspace{3cm}\\\nonumber
U_A = \{\mu\in M(\sigma_{m}(2^X))\colon \mu(O_B)\stackrel{\varepsilon'}{\sim} 1 \ \mbox{ for all } B\subseteq A\}\,. 
\end{eqnarray}
Since every $O_B$ such that $B\subseteq A$, is a clopen neighborhood of $\chi_A$ in $\sigma_m(2^X)$, the set $U_A$ is an open neighborhood of $\delta_{\chi_A}$ in $M(\sigma_m(2^X))$.
Define $V_A = \phe_T^{-1}(U_A)$. 
 Since $\chi_{A}\in V_A$, we can find a finite set $S^2_A\subseteq X\setminus A$ such that
 \begin{eqnarray}\label{eta1_2}
 W_A = \{\chi_B\in \sigma_{n}(2^X)\colon A\subseteq B \mbox{ and } B\cap S^2_A = \emptyset\}\subseteq V_A.
 \end{eqnarray}
 
 For every $A\in[X]^{\le n}\setminus [X]^{\le m}$, we put $S^2_A = \emptyset$.
 
 Let $S_A = S^1_A \cup S^2_A$, for $A\in[X]^{\le n}$. Using Lemma \ref{omega_p} we can find a
 set $Y\in [X]^{p}$ such that $Y\cap S_A = \emptyset$ for all $A\in [Y]^{\leq m}$.

It is enough to prove that $\|\phe_{T|_Y}(\chi_Z) - \phe_{T_Y}(\chi_Z)\| < \varepsilon$ for every $Z\in [Y]^{\leq n}$. So let us fix $Z\in [Y]^{\leq n}$. If $|Z|\leq m$ then we have $\phe_{T|_Y}(\chi_Z) = \phe_{T_Y}(\chi_Z) = \delta_{\chi_Z}$ and we are done. So we suppose that $m<|Z|\leq n$. Then, by the definition of $T_Y$ in Theorem~\ref{naturalextension}
$$\phe_{T_Y}(\chi_Z) = \sum_{B\in [Z]^{\leq m}} (-1)^{m-|B|}\binom{|Z| - |B|-1}{m-|B|} \delta_{\chi_B}$$
while $\phe_{T|_Y}(\chi_Z)$ has to be given in the form
$$\phe_{T|_Y}(\chi_Z) = \sum_{B\in [Y]^{\leq m}} a_B \delta_{\chi_B}$$
for certain scalars $a_B = \phe_{T|_Y}(\chi_Z)\{\chi_B\}$. Since $S^1_Z\cap Y\subset S_Z\cap Y = \emptyset$, it is clear that $|a_B|<\varepsilon'<\varepsilon /2^{p+1}$ if $B\not\subset Z$. So it is enough to show that
$$\left|a_B - (-1)^{m-|B|}\binom{|Z| - |B|-1}{m-|B|} \right| < \frac{\varepsilon}{2^{p+1}}$$
when $B\subset Z$. We shall prove the following stronger inequality by reverse induction on $|B|$, starting at $|B|=m$:

$$\left|a_B - (-1)^{m-|B|}\binom{|Z| - |B|-1}{m-|B|} \right| < \varepsilon_{|B|} = \frac{\varepsilon}{2^{(|B|+1)(p+2)}}$$

 The key property is that, since $S_A\cap Y = \emptyset$ for all $A\in [Y]^{\leq m}$ (in particular for all $A\subset [Z]^{\leq m}$) we have that $\chi_Z \in W_A$ for all $A\subset Z$. This implies that $\phe_T(\chi_Z) \in U_A$, hence
$$\phe_T(\chi_Z)(O_B) \stackrel{\varepsilon'}{\sim} 1 \text{ for all } B\in [Z]^{\leq m}$$
On the other hand, since $Y\cap S^1_Z = \emptyset$ we can write
$$(\star)\ \ 1\stackrel{\varepsilon'}{\sim} \phe_T(\chi_Z)(O_B) \stackrel{\varepsilon'}{\sim}  \phe_{T|_Y}(\chi_Z)\{\chi_C : B\subset C\in [Z]^{\leq m}\} = \sum_{B\subset C\in [Z]^{\leq m}} a_C$$

When $|B|=m$ we are done, because the left-hand side of the above expression equals $ (-1)^{m-|B|}\binom{|Z| - |B|-1}{m-|B|} $ and the right-hand side equals $a_B$. If $|B|<m$ we can apply an inductive hypothesis to all sets $C\in [Y]^{\leq m}$ with $|C|>|B|$, so
$$(\star\star) \ \ a_C \stackrel{\varepsilon_{|B|+1}}{\sim}  (-1)^{m-|C|}\binom{|Z| - |C|-1}{m-|C|} \ \ \text{ when }|C|>|B|.$$
On the other hand, by equation~(\ref{Phi1}) in Lemma~\ref{Phi}, applied to $k=m-|B|$ and $l=|Z|-m$, and $i$ running over all possible cardinalities of $C\setminus B$,
\begin{eqnarray*} (\star\star\star)\ \ 1 &=&  \sum_{i=0}^m (-1)^{m-|B|-i}\binom{|Z|-|B|}{i}\binom{|Z|-|B|- i -1}{m-|B|-i}\\
&=& \sum_{B\subset C\in [Z]^{\leq m}} (-1)^{m-|C|}\binom{|Z| - |C|-1}{m-|C|} \end{eqnarray*}
Putting $(\star)$, $(\star\star)$ and $(\star\star\star)$ together, we conclude that
$$\left|a_B - (-1)^{m-|B|}\binom{|Z| - |B|-1}{m-|B|} \right| < 2\varepsilon' +  2^p\varepsilon_{|B|+1} < 2^{p+1}\varepsilon_{|B|+1}<\varepsilon_{|B|}$$
\end{proof}

As a corollary to Theorem~\ref{somewherenatural} we obtain that, when $X$ is large enough, $T_X$ has minimal norm among all extension operators $C(\sigma_m(2^X))\To C(\sigma_n(2^X))$.

\begin{cor}
If $|X|\geq \aleph_{n-1}$, then $$\eta(\sigma_m(2^X),\sigma_n(2^X)) = \|T_X\| = \sum_{k=0}^m \binom{n}{k}\binom{n-k-1}{m-k}.$$
\end{cor}

\begin{proof}
Let $T:C(\sigma_m(2^X))\To C(\sigma_n(2^X))$ be an extension operator, and fix $\varepsilon>0$. Consider the set $Y$ given by Theorem~\ref{somewherenatural} for $p=n$. Then, $\|T_Y\| = \|T_X\|$ is the same number above, and since $\|T|_Y - T_Y\|<\varepsilon$, we have that
$\|T\| \geq \|T|_Y\| \geq \|T_Y\|-\varepsilon = \|T_X\|-\varepsilon$.  
\end{proof}

For the special case of  $n=m+1$, one can easily calculate that the above formula has much simpler form:

\begin{cor}
If $|X|\geq \aleph_{m}$, then $$\eta(\sigma_m(2^X),\sigma_{m+1}(2^X)) = 2^{m+1} -1.$$
\end{cor}

We finish this section with a remark that, although Theorem~\ref{somewherenatural} deals with sets of high cardinality, it is possible to express in terms of finite sets the fact that the extension operators $T_X$ are natural and canonical. For this, given an injective map $u:Y\To X$ and $m<\omega$, consider the operator $e_u:C(\sigma_m(2^Y))\To C(\sigma_m(2^X))$ given by $e_u(f)(\chi_A) = f(\chi_{u^{-1}(A)})$, that naturally generalizes the operators $e_Y^X$ introduced before.

\begin{thm}\label{naturalfinite}
Let $m,n$ be positive integers such that $m<n$.
Suppose that we have $M>0$, and for each finite set $X$ we have an extension operator $\tilde{T}_X:C(\sigma_m(2^X))\To C(\sigma_n(2^X))$ with $\|\tilde{T}_X\|\leq M$ in such a way that the diagram
$$
\begin{CD} C(\sigma_m(2^X)) @>\tilde{T}_X>> C(\sigma_n(2^X)) \\
 @A{e_u}AA @A{e_u}AA\\
 C(\sigma_m(2^Y)) @>\tilde{T}_Y>> C(\sigma_n(2^Y))
\end{CD}
$$
commutes for any injective map $u:Y\To X$ between finite sets. Then $\tilde{T}_X = T_X$ for all $X$.
\end{thm}

\begin{proof}
Let $W$ be any set of cardinality $\aleph_\omega$. Then by the Stone-Weierstrass Theorem $$\mathfrak{D}_W = \bigcup_{Y\in [W]^{<\omega}}e_Y^W(C(\sigma_m(2^Y))$$
is a dense subspace of $C(\sigma_m(2^W))$. The function $$\tilde{T}_W: \mathfrak{D}_W \To C(\sigma_n(2^W))$$
given by $\tilde{T}_W(e_Y^W(f)) = e_Y^W(\tilde{T}_Y(f))$ is well defined independently of the choice of $Y\in [W]^{<\omega}$, and is a linear function with $\|\tilde{T}_W\|\leq M$, so it extends to a globally defined extension operator
$$\tilde{T}_W: C(\sigma_m(2^W)) \To C(\sigma_n(2^W))$$
Now, we consider any finite set $X$, $\varepsilon>0$, and we shall check that $\|\tilde{T}_X - T_X\|<\varepsilon$. By Theorem~\ref{somewherenatural}, we can find $Y\in [W]^{|X|}$ such that $\|\tilde{T}_Y - T_Y\| = \|\tilde{T}_W|_Y - T_Y\|<\varepsilon$. But, if $u:Y\To X$ is a bijection,  the diagram in the statement of Theorem~\ref{naturalfinite} commutes for both the operators $\tilde{T}_\ast$ and $T_\ast$, so the inequality is transferred to $\|\tilde{T}_X - T_X\|<\varepsilon$.
\end{proof}

\section{Extension operators on spaces $C(\sigma_m(2^{\aleph_1}))$}

\begin{thm}\label{omega1}
If $|X|=\aleph_1$, then $\eta(\sigma_m(2^X),\sigma_n(2^X)) = 2n-2m+1$.
\end{thm}
\begin{proof}

Inequality $[\leq]$. Suppose $X=\omega_1$, let $<$ be the usual order of ordinals, and for each $\beta<\omega_1$, let $<_\beta$ (the $\beta$-order) be an order on $\beta$ of order type $\omega$. Given a set $A = \{\alpha_1<\alpha_2<\ldots<\alpha_k\}\subset \omega_1$ and $i\in\{m,\ldots,k\}$, let $\Gamma_i(A)$ be the set consisting of the first $m-1$ elements of $\{\alpha_1,\ldots,\alpha_{i-1}\}$ according to the $<_{\alpha_i}$-order.

We define an extension operator $T:C(\sigma_m(2^X))\To C(\sigma_n(2^X))$ by the formula $$Tf(\chi_A) = \sum_{i=m}^k f(\chi_{\{\alpha_i\}\cup \Gamma_i(A)}) - \sum_{i=m+1}^k f(\chi_{\Gamma_i(A)})$$
when $A=\{\alpha_1<\alpha_2<\ldots<\alpha_k\}$ is as above for some $k>m$ (and, of course, $Tf(\chi_A) = f(\chi_A)$ if $|A|\leq m$. The only point to be checked is that $Tf$ is indeed a continuous function on $C(\sigma_n(2^X))$ whenever $f\in C(\sigma_m(2^X))$ is continuous, because once this is established it is straightforward that $T$ is linear and $\|T\| = 2n-2m+1$.


So we fix $f\in C(\sigma_m(2^X))$. Since $\sigma_m(2^X)$ is a Fr\'{e}chet-Urysohn space, it is enough to check the sequential continuity of $Tf$. So suppose that we have a sequence $\{\chi_{A^\xi}\}_{\xi<\omega}$ that converges to $\chi_A$ in $\sigma_m(2^X)$. 
By passing to a subsequence, we can suppose that our sequence $\{\chi_{A^\xi}\}_{\xi<\omega}$ has the following homogeneity properties:

\begin{itemize}

\item All $A^\xi$ have the same cardinality $p$ (and $p>m$, otherwise it is obvious) and we write $A^\xi = \{\alpha_1^\xi<\alpha_2^\xi<\cdots<\alpha_p^\xi\}$.

\item The orders $<_{\alpha_i^\xi}$ behave homogeneously on all $A^\xi$. That is, for $\xi,\zeta<\omega$, and for $u,v<i\leq p$,
$$\alpha_u^\xi <_{\alpha_i^\xi} \alpha_v^\xi \iff \alpha_u^\zeta <_{\alpha_i^\zeta} \alpha_v^\zeta$$

\item The limit set $A$ is of the form $A=\{\alpha_{i[1]}<\cdots<\alpha_{i[k]}\}$ where each $\{\alpha^\xi_{i[j]}\}_{\xi<\omega}$ is a constant sequence equal to $\alpha_{i[j]}$, while for other $j$'s, the sequence $\{\alpha_j^\xi\}_{\xi<\omega}$ has no infinite repetitions. 
 
\end{itemize} 
 
 For the last two condition to happen, it must be the case that, for each $j$ and $\xi$, the set $\{\alpha_{i[1]},\ldots,\alpha_{i[j-1]}\}$ is an initial segment of $\{\alpha^\xi_1,\ldots,\alpha^\xi_{i[j]-1}\}$ in the $<_{\alpha_{i[j]}}$-order, and this makes the operation $\Gamma_i$ to behave nicely. We can now compute where $Tf(\chi_{A^\xi})$ converges (explanations are given below):

\begin{eqnarray*}
Tf(\chi_{A^\xi}) &=& \sum_{i=m}^p f(\chi_{\{\alpha_i^\xi\}\cup \Gamma_i(A^\xi)} )- \sum_{i=m+1}^p f(\chi_{\Gamma_i(A^\xi)})\\
&\xrightarrow{\xi\rightarrow\infty}& \sum_{i= m}^p f(\chi_{A\cap (\{\alpha_i\}\cup \Gamma_i(A^\xi))}) - \sum_{i=m+1}^p f(\chi_{A\cap \Gamma_i(A^\xi)})\\
&=& f(\chi_{A\cap (\{\alpha_m\}\cup \Gamma_m(A^\xi))}) + \sum_{i[j]>m}f(\chi_{\{\alpha_{i[j]}\}\cup \Gamma_j(A)}) - f(\chi_{\Gamma_j(A)})\\
\end{eqnarray*}

One remark is that $\xi$ seems to wrongly remain as a parameter after taking limits on $\xi$, but the point is that by the homogeneity properties assumed for the $A^\xi$'s, the second and third line expressions above are indeed independent on the choice of $\xi$. The last equality is because, on the one hand $\Gamma_{i[j]}(A^\xi) = \Gamma_j(A)$, for $i[j]>m$, due to the initial segment property stated just before the computation, and on the other hand, the terms corresponding to indexes $\alpha_i$ with $i\not\in\{m\}\cup \{i[j]\}_{j=1,\ldots,k}$ cancel on both sides.

Now, focusing on the expression that we obtained in the third line above, in the second sum, the summands for $j\leq m$ cancel in telescoping sum. Let us consider $i[a]$ and $i[b]$ the first and last index among the $i[j]$'s such that $i[j]>m$ and $j\leq m$, if there are any such indexes at all (we consider the other case later). We obtain:

\begin{eqnarray*}
Tf(\chi_{A^\xi}) &\longrightarrow& f(\chi_{A\cap (\{\alpha_m\}\cup \Gamma_m(A^\xi))} )+ f(\chi_{\{\alpha_{i[1]},\ldots,\alpha_{i[b]}\}}) - f(\chi_{\{\alpha_{i[1]},\ldots,\alpha_{i[a-1]}\}}) \\ &+& \sum_{j>m}f(\chi_{\{\alpha_{i[j]}\}\cup \Gamma_j(A)}) - f(\chi_{\Gamma_j(A)})\\
\end{eqnarray*}

But now, we observe that the first and third term are just the same thing and they cancel, so we get

\begin{eqnarray*}
Tf(\chi_{A^\xi}) &\longrightarrow& f(\chi_{\{\alpha_{i[1]},\ldots,\alpha_{i[b]}\}}) + \sum_{j>m}f(\chi_{\{\alpha_{i[j]}\}\cup \Gamma_j(A)}) - f(\chi_{\Gamma_j(A)})\\
\end{eqnarray*}

which is precisely $Tf(\chi_A)$, because $\{\alpha_{i[1]},\ldots,\alpha_{i[b]}\} = A$ if $|A|\leq m$, while $b=m$ if $|A|>m$. We were left the case when there were no indexes such that $i[j]>m$ and $j\leq m$. This means that $\alpha_{i[j]} = \alpha_j$ for $j\leq m$, and we get again the expression of $Tf(\chi_A)$.

Inequality $[\geq]$. Fix $\varepsilon>0$, $X$ uncountable, and $T:C(\sigma_m(2^X))\To C(\sigma_n(2^X))$ an extension operator, and we shall prove that $\|T\| \geq 2n-2m+1-\varepsilon$. We proceed as in the proof of Theorem~\ref{somewherenatural}. Pick a much smaller $\varepsilon'< 6n\varepsilon$. Consider the measure-valued function $\phe_T:\sigma_n(2^X) \To \|T\| M(\sigma_m(2^X))$ associated with the operator $T$, given by $\phe_T(\chi_A) = T^\ast(\delta_{\chi_A})$.

For every $A,B\in [X]^{\le m}$ define
\begin{eqnarray}\label{eta1_1}
O_B = \{\chi_C\in \sigma_{m}(2^X)\colon B\subseteq C \} \,,\hspace{3cm}\\\nonumber
U_A = \{\mu\in M(\sigma_{m}(2^X))\colon \mu(O_B)\stackrel{\varepsilon'}{\sim} 1 \ \mbox{ for all } B\subseteq A\}\,. 
\end{eqnarray}
Since every $O_B$ such that $B\subseteq A$, is a clopen neighborhood of $\chi_A$ in $\sigma_m(2^X)$, the set $U_A$ is an open neighborhood of $\delta_{\chi_A}$ in $M(\sigma_m(2^X))$.
Define $V_A = \phe_T^{-1}(U_A)$. 
 Since $\chi_{A}\in V_A$, we can find a finite set $S_A\subseteq X\setminus A$ such that
 \begin{eqnarray}\label{eta1_2}
 W_A = \{\chi_B\in \sigma_{n}(2^X)\colon A\subseteq B \mbox{ and } B\cap S_A = \emptyset\}\subseteq V_A.
 \end{eqnarray}
 Using Lemma \ref{omega_1} we can find a
 set $Z = \{z_1,\ldots,z_n\}\in [X]^{n}$ such that $z_j\not\in S_{\{z_i : i\in I\}}$ whenever $j<\min(I)$ or $j>\max(I)$. We shall check that $\mu = \phe_T(\chi_Z)$ satisfies $\|\mu\| \geq 2n-2m+1-\varepsilon$, which finishes the proof. For $i=1,\ldots,n-m+1$ let $A_i = \{z_i,\ldots,z_{i+m-1}\}$. Notice that, by the key property of $Z$, $Z \cap S_{A_i} = \emptyset$, hence $\chi_{Z}\in W_{A_i}$, hence $\mu = \phe_T(Z)\in U_{A_i}$,  hence 
$$(\star) \ \ \mu(O_{A_i}) = \mu(\{\chi_{A_i}\})\stackrel{\varepsilon'}{\sim} 1 \ \ \text{ for } i=1,\ldots,n-m+1$$
Now, for $i=2,\ldots,n-m+1$, consider $\tilde{A}_i = \{z_i,\ldots,z_{i+m-2}\}$. Again, we have that  $Z \cap S_{\tilde{A}_i} = \emptyset$, so by the same argument as above,
$$(\star\star)\ \ \mu(O_{\tilde{A}_i})  \stackrel{\varepsilon'}{\sim} 1 \ \ \text{ for } i=2,\ldots,n-m+1$$
Using $(\star)$ and $(\star\star)$ we conclude that
$$(\star\star\star)\ \ \mu(O_{\tilde{A}_i}\setminus \{\chi_{A_{i-1}},\chi_{A_i}\})  \stackrel{3\varepsilon'}{\sim} -1 \ \ \text{ for } i=2,\ldots,n-m+1$$
All sets appearing in $(\star)$ and $(\star\star\star)$ are pairwise disjoint, and there are $2n-2m+1$ of them, so we conclude that 
$$\|\mu\| > 2n-2m+1 - (2n-2m+1)3\varepsilon'$$
By the choice of $\varepsilon'$ we are done.
\end{proof}

\section{Balls of the Hilbert space}

This section is devoted to the proof of Theorem~\ref{balls}. First, we can get a simpler topologically equivalent description of the balls in Hilbert space. We fix an uncountable set $X$. For $\lambda\in (0,+\infty)$, let $B_\lambda = \{z\in\mathbb{R}^X : \sum_{i\in X}|z_i|\leq\lambda\}$. If $\Delta:\ell_2(X) \To \mathbb{R}^X$ is the function given by $\Delta((z_i)_{i\in X}) = (sign(z_i)\cdot z_i^2)_{i\in X}$, then it is easy to check that $\Delta: \lambda B_{\ell_2(X)} \To B_{\sqrt{\lambda}}$ establishes a homeomorphism for each $\lambda$. Thus, Theorem~\ref{balls} can be equivalently reformulated saying that for each $0<\lambda<\mu<+\infty$, there is no extension operator $T:C(B_\lambda)\To C(B_\mu)$. We can also look at the compact sets $B^+_\lambda = \{z\in B_\lambda : \forall i\in X\ z_i\geq 0\}$. There is a continuous retraction $\rho:B_\lambda\To B_\lambda^+$ given by $\rho((z_i)_{i\in X}) = (|z_i|)_{i\in X}$. If an extension operator $T:C(B_\lambda)\To C(B_\mu)$ exists, then $T^+(f) = T(f\circ\rho)|_{B^+_\mu}$ would give an extension operator $T^+:C(B^+_\lambda)\To C(B^+_\mu)$. Thus, it is enough to prove the following:

\begin{thm}\label{1balls}
For each $0<\lambda<\mu<+\infty$, there is no extension operator $T:C(B^+_\lambda)\To C(B^+_\mu)$.
\end{thm}

For $n<\omega$, Let $S_n$ be the set of all elements of $B^+_1$ whose coordinates take the only values 0 or $1/n$. The set $S_n$ is homeomorphic to $\sigma_n(2^X)$. It is a standard terminology to call an extension operator $T$ to be \emph{regular} if it $T$ is positive, $\|T\|=1$, and it preserves constant functions. The only fact among these that is relevant to us is the value of the norm.

\begin{lem}\label{reg}
For each $m<\omega$, there exists a regular extension operator  $$R:C(S_m)\To C(B^+_1).$$
\end{lem}

\begin{proof}
Fix $\epse$ such that $1/(m+1)<\epse<1/m$. Choose a continuous nondecreasing function $g_1:\mathbb{R}\To [0,1]$ such that $g_1(t) = 0$ if $t\leq \epse$ and $g(t) = 1$ if $t\geq 1/m$. Let also $g_0 = 1 - g_1$. For $z=(z_i)_{i\in X}\in B_1$, define the set $$F_z = \{i\in X : z_i>\epse\}\in [X]^{\leq m}.$$ The operator $R$ is given by the following formula:

$$Rf(z) = \sum_{A\subset F_z} \left(\prod_{i\in F_z} g_{\chi_A(i)}(z_i)\right)\cdot f(m^{-1}\chi_A)$$
(here, we use the convention that $\prod_{i\in F_z} g_{\chi_A(i)}(z_i)=1$, if $F_z=\emptyset$).

First, we check that $Rf$ is an extension of $f$. Namely, if $z= m^{-1}\chi_A \in S_m$, then $F_z = A$ and  $g_{\chi_A}(i)(z_i)$ equals 0 when $i\not\in A$ and equals 1 when $i\in A$. So $Rf(z) = f(z)$. Second, for any $z\in B^+_1$, since $g_0(z_i) + g_1(z_i) = 1$, it is easy to check  that 
\begin{eqnarray*}
\sum_{A\subset F_z} \left(\prod_{i\in F_z} g_{\chi_A(i)}(z_i)\right) &=& 1\\
\end{eqnarray*}
From this equality, it easily follows that $\|R\|=1$, and, by the way, that $R$ preserves constant functions.

It remains to verify that $Rf$ is indeed continuous whenever $f$ is continuous. First, we shall check this when $f$ depends on finitely many coordinates, i.e., there is a finite set $H\subset X$ such that $f(m^{-1}\chi_A)= f(m^{-1}\chi_{A\cap H})$, for any $A\in [X]^{\leq m}$. Then, using the equality $g_0(z_i) + g_1(z_i) = 1$, one can verify that

$$Rf(z) = \sum_{A\subset F_z\cap H} \left(\prod_{i\in F_z\cap H} g_{\chi_A(i)}(z_i)\right)\cdot f(m^{-1}\chi_{A}).$$

This formula shows that $Rf$ is continuous.

Now, observe that by the Stone-Weierstrass Theorem, the family $\mathcal{F}$ of functions depending on finitely many coordinates is dense in $C(S_m)$, therefore any function $f\in C(S_m)$ is a uniform limit of a sequence $(f_n)$ of functions $f_n\in \mathcal{F}$. Since $\|R\|=1$, $Rf$ is a uniform limit of a sequence $(Rf_n)$, which demonstrates the continuity of $Rf$.
\end{proof}

We proceed to the proof of Theorem~\ref{1balls}. It is enough to prove the case when $\lambda=1$, and later apply the homeomorphism $(z_i)_{i\in X}\mapsto (\lambda^{-1}z_i)_{i\in X}$. So suppose that there exists such an operator $T:C(B^+_1)\To C(B^+_\mu)$. Take two natural numbers $m,k$ such that $m>k>\|T\|$ and $1 + k/m < \mu$. Then, we have a diagram of inclusions

$$\begin{CD}
 B^+_1 @>>> B^+_\mu \\
 @AAA @AAA \\
 S_n @>>> \frac{m+k}{m}S_{m+k} .
\end{CD}$$

By Lemma~\ref{reg}, the left vertical arrow has a regular extension operator $R:C(S_m)\To C(B^+_1)$. Therefore, we have an extension operator for the lower arrow, namely $E = rTR$, where $r$ is the restriction operator. Notice that $\|E\| \leq \|T\|$. But notice that the lower arrow is just the same as the inclusion $\sigma_m(2^X) \To \sigma_{m+k}(2^X)$, hence by Theorem~\ref{omega1} and Proposition~\ref{smallerset}, we get $\|T\| \geq \|E\| \geq 2k+1 >\|T\|$, a contradiction.


\begin{thebibliography}{BM}

\bibitem[AA]{AA} S. Argyros, A. Arvanitakis, \textit{ Regular
averaging and regular extension operators in weakly compact
subsets of Hilbert spaces}, Serdica Math. J. \textbf{ 30} (2004),
527--548.


\bibitem[AK]{AK} A. Aviles, P. Koszmider, A continuous image of a Radon-Nikodym compact
space which is not Radon-Nikodym, Duke Math. J. {\bf 162}, 12 (2013), 2285--2299.

\bibitem[BM]{BM}
M. Bell and W. Marciszewski, \emph{On scattered Eberlein compact
spaces}, Israel J. Math. \textbf{158} (2007), 217--224.

\bibitem[Be]{Be} Y. Benyamini, \emph{Constants of simultaneous
extension of continuous functions}, Israel J. Math. \textbf{16}
(1973), 258--262.

\bibitem[CL]{CL} H. H. Corson and J. Lindenstrauss, \emph{On simultaneous
extension of continuous functions}, Bull.\ Amer.\ Math.\ Soc.\
\textbf{71} (1965), 258--262.

\bibitem[ER]{ER} P. Enflo and H. P. Rosenthal, \emph{Some results concerning
$L^p(\mu)$-spaces}, J. Functional Analysis \textbf{14} (1973),
325--348.

\bibitem[Ko]{Ko} P. Koszmider, \textit{ Banach spaces of continuous functions with few operators}, Math. Ann. \textbf{ 330}
(2004), 151--183.

\bibitem[LT]{LT}  J. Lindenstrauss and L. Tzafriri, \emph{Classical Banach spaces.} Lecture Notes in Mathematics, Vol. 338. Springer-Verlag, Berlin-New York, 1973.

\bibitem[Ma]{Ma} W. Marciszewski, \emph{On Banach spaces $C(K)$
isomorphic to $c_0(\Gamma)$}, Studia Math. \textbf{156} (2003),
295--302.

\bibitem[Pl]{Pl} G. Plebanek, \textit{A construction of a Banach space $C(K)$ with few operators},         Topology Appl. \textbf{ 143}
(2004), 217-239.

\bibitem[Se]{Se}
Z. Semadeni, \textit{Banach Spaces of Continuous Functions}
 PWN, Warsaw, 1971.

\end{thebibliography}
\end{document}